\documentclass[12pt, reqno]{amsart}
\usepackage{amsmath,amsfonts,amssymb,mathabx,latexsym,mathtools}
\usepackage[colorinlistoftodos]{todonotes}
\usepackage{enumitem}
\usepackage{comment}
\usepackage{shuffle}
\usepackage{hyperref}
\usepackage{tikz}
\usetikzlibrary{intersections}
\usepackage{soul}
\usepackage[backend=bibtex]{biblatex}
\usepackage[capitalize,noabbrev]{cleveref}
\newtheorem{theorem}{Theorem}[section]
\newtheorem{lemma}[theorem]{Lemma}
\newtheorem{proposition}[theorem]{Proposition}
\newtheorem{corollary}[theorem]{Corollary}
\newtheorem{remark}[theorem]{Remark}
\newtheorem{conjecture}[theorem]{Conjecture}
\numberwithin{equation}{section}
\newtheorem*{example}{Example}

\newcommand{\Sym}{\ensuremath{\mathfrak{S}}}
\newcommand{\Z}{\ensuremath{\mathbb{Z}}}

\newcommand{\N}{\ensuremath{\mathbb{N}}}
\newcommand{\E}{\ensuremath{\mathbb{E}}}
\newcommand{\bigmid}{\ensuremath{\bigg \lvert}}
\newcommand{\cgeq}[1]{ \underset{(#1)}{>} } 
\newcommand{\prob}[1] { \mathbb{P} \left( #1 \right) }
\newcommand{\maj}{\mathrm{maj}}
\newcommand{\cov}{\mathrm{cov}}
\newcommand{\Var}{\mathrm{Var}}
\newcommand{\LIS}{\mathrm{LIS}}
\newcommand{\Maj}{\mathrm{Maj}(n,q)}
\newcommand{\Maji}{\mathrm{Maj}(n,1/q)}
\newcommand{\geo}[1]{\mathrm{Geom}(#1)}

\title{Major index distribution}
\author[M. Coopman]{Michael Coopman}
\address[MC]{Department of Mathematics, University of Florida, Gainesville, FL 32601}
\email{m.coopman@ufl.edu}

\begin{document}
\begin{abstract}
For $0<q<1$, let $\Maj$ be the distribution on the symmetric group $\Sym_n$ such that a permutation $\pi \in \Sym_n$ is selected with probability proportional to $q^{\maj(\pi)}$. 
The distribution has connections to $q$-Plancherel measure.
We describe an algorithm that realizes $\Maj$, and use it to prove known results of $q$-Plancherel measure without the need of representation theory. 
This sampler is transparent and elegant, allowing properties of $\Maj$ about its limit shape, pattern normality, and cycle structure to be obtained.
\end{abstract}

\maketitle

\section{Introduction}

\begin{figure}
    \centering
    \includegraphics[width=4.5cm]{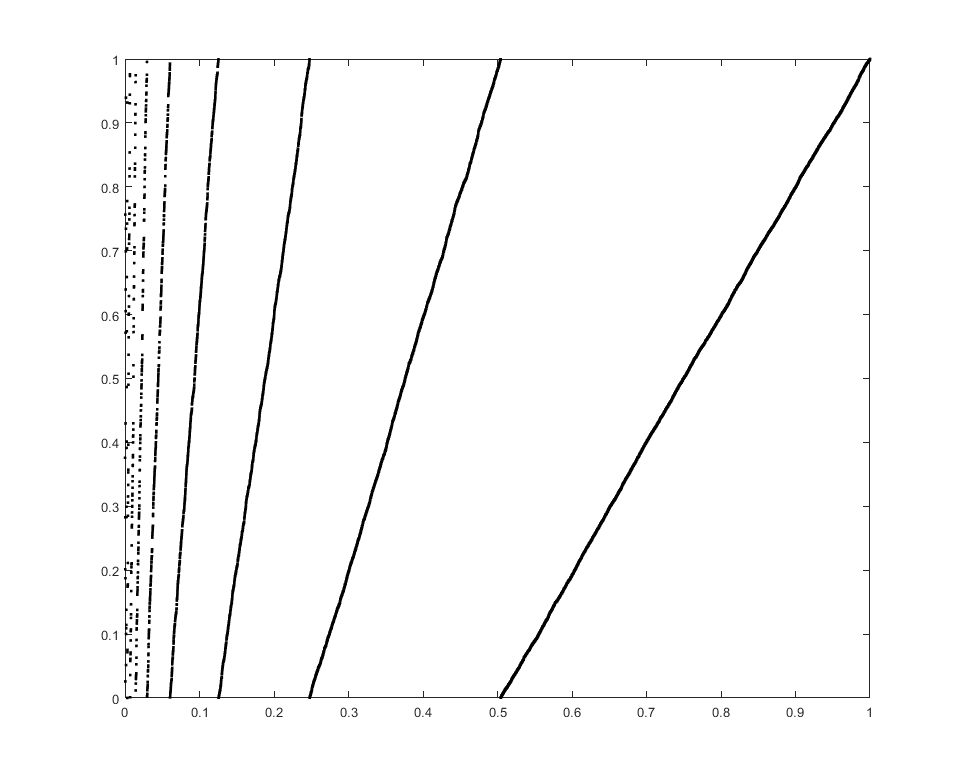}
    \includegraphics[width=4.5cm]{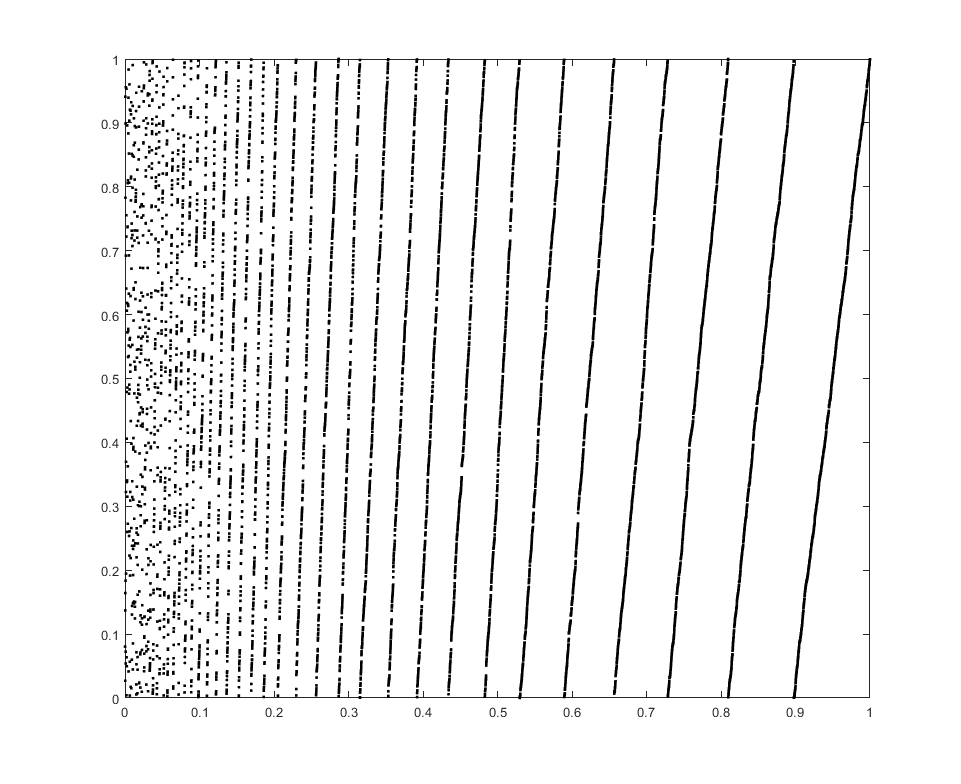}
    \caption{Two typical permutations taken from $\Maj$. The left picture is for $\mathrm{Maj}(10000,0.5)$. The right picture is for $\mathrm{Maj}(10000,0.9)$.}
    \label{fig:graph}
\end{figure}

Random permutations are ubiquitous in combinatorics, probability, and statistics with applications to all quantitative disciplines.
In many of these applications, the distributions of these permutations are not uniform.
Some distributions assign each permutations $\pi$ a probability proportional to $q^{f(\pi)}$, where $q > 0$ and $f(\pi)$ is some statistic on the symmetric group.
 When the statistic is inversion count, we recover the Mallows distribution, which was first used to answer questions in statistical ranking models ~\cite{mallows1957non} and more recently in genomics ~\cite{fang2021arcsine}.
When the statistic is cycle count, we recover the Ewens distribution, which was first used to answer questions in neutral allele theory ~\cite{ewens1972sampling}.

This paper will focus on the case where the statistic is the major index. This gives us the major index distribution, and we will answer questions regarding $q$-Plancherel measure with it.
For $\pi \in \Sym_{n}$, the \emph{major index} $\maj(\pi)$ of a permutation $\pi$ is the sum of its descents.
For $q > 0$, we say that $\Maj$ follows the \emph{major index distribution} \footnote{There is another distribution related to major index introduced by Fulman ~\cite{fulman1999probabilistic}.
We will not cover that distribution in this paper.} if for each permutation $\pi \in \Sym_{n}$,
\[
\mathbb{P}(\Maj = \pi) = \frac{q^{\maj \pi}}{[n]_q!}.
\]

$\Maj$ is closely related to \emph{$q$-Plancherel measure}, a distribution on partitions {that can be seen as} a deformation of the standard Plancherel measure.
Strahov ~\cite{strahov2008differential} observed that the connection of $q$-Plancherel measure to major index distribution through RSK {is the same as that of} Plancherel measure to the uniform distribution.
This relationship was further noted by F\'eray and M\'eliot ~\cite{feray2012asymptotics} who used representation theory to prove asymptotics and normality results about the limiting shape of $q$-Plancherel measure.
Their results translates to statements about the length of the longest increasing subsequence of the major index distribution.

Our goal is to understand major index distribution.
The main tool of this paper is the introduction a novel sampler for the distribution.
During the preparation of this manuscript, it came to the author's attention that this sampler originally appeared in ~\cite{stanley2001generalized} in the context of card shufflings.

With it, we obtain limiting behavior, recover results of $q$-Plancherel measure in ~\cite{feray2012asymptotics}, and derive {statistical properties} of the distribution.

\subsection{Main Results}

For $q \in (0,1)$, $\Maj$ can be sampled using $n$ i.i.d.\! geometric random variables with parameter $1-q$.
For $q > 1$, $\Maj$ can be obtained by taking the height complement of $\Maji$.

Let $\Gamma_{n}: \mathbb{N}^n \rightarrow \Sym_{n}$ be defined as follows.
For $G = (G_{1}, \hdots, G_{n})$, let $\preceq$ be a total ordering on the pairs $(G_i, i)$ such that $(G_i, i) \preceq (G_j, j)$ if one of the conditions hold:
\[
G_i > G_j \qquad \mbox{or} \qquad G_i = G_j \mbox{ and } i < j.
\]
This ordering induces an ordering on $[n]$ by ignoring the first element of each pair.
This new ordering on $[n]$, when read in ascending order, is the permutation $\Gamma_{n}(G)$.

\begin{example}
Let $G = (5, 2, 3, 3, 0, 1, 6)$
The $(G_i,i)$ pairs in ascending order would be $(6,7)$, $(5,1)$, $(3,3)$, $(3,4)$, $(2, 2)$, $(1,6)$, and $(0,5)$.
Based on this ordering, the permutation would be $7134265$.
\end{example}

\begin{theorem}
\label{t:introsampler}
    Let $G = (G_{1}, \hdots, G_{n})$ be a sequence of i.i.d. geometric random variables with parameter $1-q$.
    Then, $\Gamma(G)$ follows the major index distribution. 
\end{theorem}
The self-contained proof can be found in Section 3.

As shown in Figure ~\ref{fig:graph}, $\Maj$ possesses a very rigid structure.
The curves that appear can be viewed as lattice paths $L_1, L_2, \hdots$ indexed from right to left.
This notion is made precise in Section 4.

\begin{theorem}
\label{p:introApplyDonsker}
    Let $L_{i}$ be the $i$-th lattice walk mentioned above for $\Maj$. 
    For all $i > 0$ and as $n \rightarrow \infty$,
    \[
        \left( \frac
        {L_{i}(\lfloor nt \rfloor) - p_{i}t - L_{i}(0)}
        {\sqrt{np_{i}(1-p_{i})}} 
        \right)_{0 \leq t \leq 1} \overset{dist.}{\longrightarrow} (W(t))_{0 \leq t \leq 1},
    \]
    where $W(t)$ is standard Brownian motion and $p_i = q^{i}(1-q)$.
 \end{theorem}

Define $\lambda(\pi) = (\lambda_1(\pi), \lambda_2(\pi), \hdots)$ such that $\lambda_{1}(\pi) + \lambda_{2}(\pi) + \hdots + \lambda_{i}(\pi)$ is the maximum total length amongst all $i$-tuples of disjoint increasing subsequences of $\pi$.
Note $\lambda_1(\pi) = \LIS(\pi)$, the length of the longest increasing subsequence of $\pi$.
The \emph{$q$-Plancherel measure} is $\lambda(\Maj)$.
This measure was first studied by Kerov in ~\cite{kerov1992qhook} as a deformation of the Plancherel measure by modifying the standard hook walk algorithm.
 
Heuristically, $\lambda_{i}(\pi)$ is the number of elements forming $L_{i}$ with a negligible amount of discrepancies.
Specifically, $\E[\lambda_{i}(\pi)] = |L_{i}| + o(n^{\epsilon}\log(n))$ for arbitrary $\epsilon > 0$. 
As a consequence, we recover the following result of ~\cite{feray2012asymptotics}.
\begin{theorem}[Theorem 2 of ~\cite{feray2012asymptotics}]
\label{intro:feray}
Let $q > 0$. Define 
\[
Y_{n,i} = \sqrt{n} \left( \frac{\lambda_{i}}{n} - q^{i-1}(1-q) \right).
\]
Then, the random process $(Y_{n,i})_{i \geq 1}$ converges to a Gaussian process $(Y_{i})$ with
\begin{align}
    \E[Y_{i}] & = 0,\\
    \E[Y_{i}^{2}] & = (1-q)q^{i-1} - (1-q)^2q^{2(i-1)},\\
    \cov(Y_{i} , Y_{j}) & = -(1-q)^{2} q^{i+j-2}.
\end{align}
\end{theorem}
The proof is in Section 7.

We say the $k$-tuple of distinct indices $I = (i_1 < i_2 < \hdots < i_{k})$ forms the \emph{pattern} $\sigma \in \Sym_{k}$ in $\pi$ if $\pi(I)$ is order-isomorphic to $\sigma$.  
For major index distribution, we find the mean and order of the variance for the number of occurrences of any pattern. As a result, the normality for these pattern frequencies is established.

\begin{theorem}
    \label{t:introNormality}
    For fixed $q>0$, permutation patterns in $\Maj$ is asymptotically normal as $n \rightarrow \infty$.
\end{theorem}

The proof relies on a theorem of Janson ~\cite{janson2018renewal} regarding U-statistics and the fact that $\maj$ is a shuffle-compatible statistic ~\cite{stanley1972ordered}.
The background is discussed in Section 2.5, while the proof is in Section 5.

Let $c_k(\pi)$ be the number of k-cycles of a permutation $\pi$.
It is well-known that for uniformly random permutations, each $c_k$ is asymptotically Poisson.
We find the mean and variance  of fixed points $c_{1}(\pi)$ as well as the mean of $c_2(\pi)$.
\begin{theorem}
    \label{t:introFixedPoints}
    For $0<q<1$ and $n \in \mathbb{N}$, we have
    \begin{align*}
        \E[c_1(\pi)] = & \sum_{\ell = 1}^{n} \frac{(1-q)^{\ell-1}}{[\ell]_q},\\
        \mathrm{Var}[c_{1}(\pi)] \underset{n \rightarrow \infty}{\longrightarrow} & \sum_{\ell \geq 1} \frac{\ell(1-q)^{\ell-1}}{[\ell]_q},\\
        \E[c_2(\pi)] = & \sum_{\ell = 1}^{\lfloor \frac{n}{2} \rfloor}  \frac{q^{\ell}(1-q)^{2\ell-2}}{[\ell]_q [2\ell]_q}.
    \end{align*}
\end{theorem}
Notably, this theorem implies that fixed points in our framework is not Poisson. 
The proof is in Section 6.

\noindent\textbf{Acknowledgements:} I am grateful to Valentin F\'eray, Jason Fulman, and Sumit Mukherjee for helpful conversations. I am especially grateful to my advisor Zachary Hamaker for extensive guidance on preparing the manuscript.

\section{Background and notation}

\subsection{Geometric random variables}

Given a random variable $X: \N \rightarrow \mathbb{R}$, we will say that $X$ \emph{follows a geometric law}, denoted $X \sim \geo{1-q}$ if $Pr(X = k) = q^k(1-q)$.
The relevant properties of geometric random variables for this article are provided below. For further background, see ~\cite{durrett2019probability, khoshnevisan2007probability}.

\begin{proposition}[Memorylessness]
\label{prop:memory}
Let $X \sim \geo{1-q}$. Then, 
\[
\prob{X > m \mid X \geq k} = \prob{X + k > m}.
\]
\end{proposition}

For our purposes, we want a more general version of memorylessness, which can be proven is a similar fashion as Proposition ~\ref{prop:memory}.

\begin{corollary}[Memorylessness II]
\label{prop:memory2}
Let $X, Y$ be random variables such that $X$ is independent from $Y$. Let $E(X, Y) \subset \Omega$ be an event that depends on random variables $X$ and $Y$. If $X \sim \geo{1-q}$, then
\[
    \prob{E(X, Y) \mid X \geq k} = \prob{E(X+k,Y)}.
\]
\end{corollary}

In Section 3, we will perform comparisons between geometric random variables. As such, the following facts will be used.

\begin{proposition}[Geometric races]
\label{prop:race}
Let $X$ and $Y$ be independent random variables such that $X \sim \geo{1-a}$ and $Y \sim \geo{1-b}$.
Then,
\begin{enumerate}
	\item $\min(X,Y) \sim \geo{1-ab}$
	\item $\prob{X \geq Y} = \dfrac{1-b}{1-ab}$
	\item $\prob{X > Y} = \dfrac{a(1-b)}{1-ab}$
\end{enumerate}
\end{proposition}

\begin{proof}
    (i) Let $Z = \min(X,Y)$. For $k \leq 0$,
    \[
        \prob{Z \geq k} = \prob{X \geq k}\prob{Y \geq k} = a^kb^k = (ab)^k.
    \]
    Therefore, $\prob{Z = k} = (ab)^k(1-ab)$.
    
    (ii) A straightforward calculation gives the following:
    \begin{align*}
        \prob{X \geq Y} & = \sum_{x=0}^{\infty} \sum_{y=0}^{x} a^{x}(1-a)b^{y}(1-b)\\
        & = \sum_{x=0}^{\infty} a^{x}(1-a) (1-b^{x+1})\\
        & = \sum_{x=0}^{\infty} a^{x}(1-a) - (ab)^{x}b(1-a)\\
        & = \frac{1-b}{1-ab}.
    \end{align*}

    (iii) A similar argument as (ii) applies for (iii).
\end{proof}

Finally, we need to use several geometric random variables at once.
Take $G \sim \geo{1-q}^{n}$ to mean that $G = (G_{i})_{i=1}^{n}$, where each $G_{i}$ is i.i.d.\!\! with distribution $\geo{1-q}$.

\subsection{Lattice walks, Brownian motion, and Donsker's theorem}

Defining $[0,n] = {0} \cup [n]$, a \emph{lattice path} $S_{n}$ is a function $S_{n}: [0,n] \rightarrow \Z$ such that
\[
    |L(i) - L(i-1)| \in \Z.
\]
All lattice paths in this paper will follow the stronger condition: $L(i) - L(i-1) \in \{0,1\}$. 
We can also view $L$ as a subset in $\Z^2$.
Starting at $(L(0),0)$, the lattice path moves towards $(L(n),n)$ using up steps $(0, 1)$ and upright steps $(1,1)$.
See Figure \ref{fig:latticePaths} for an example.

A \emph{simple random walk} $S$ is a function $S: \N \rightarrow \Z$ such that the variables $X_{i} = S(i) - S(i-1) \in \Z$ are i.i.d.\! random variables with $X_{i} = 1$ with probability $p$ and 0 otherwise. 
Clearly, $S|_{[n]}$ is a lattice path.
Under suitable scaling, simple lattice walks converge in distribution to Brownian motion. This is stated in the theorem below.
\begin{theorem}[Donsker's Theorem, Th 8.1~\cite{billey2020asymptotic}]
\label{p:applyDonsker}
Let $X_1$, $X_2$, $X_3$, $\hdots$ be i.i.d.\! random variables with mean 0 and variance 1.
Let $S_{n} = \sum_{i = 1}^{n} X_{i}$.
Then,
\[
    \left( \frac{S_{n}(\lfloor nt \rfloor)}{\sqrt{n}} \right)_{0 \leq t \leq 1} \overset{\text{dist.}}{\longrightarrow} (W(t))_{0 \leq t \leq 1},
\]
where $W(t)$ is standard Brownian motion.
\end{theorem}

We also need the rate of convergence found in ~\cite{MR666546}.

\begin{theorem}[Th 2.1.2~\cite{MR666546}]
    \label{p:donskerBounds}
    Assuming the hypothesis of Theorem ~\ref{p:applyDonsker} and that $\E[X_{i}^{3}] < \infty$, then there exists a standard Brownian motion $W(t)$ and a sequence of functions $\hat{S}_{n}(t)$ such that
    \[
    \hat{S}_{n}(t) \overset{\mbox{dist.}}{=} \frac{S_{n}(\lfloor nt \rfloor)}{\sqrt{n}},
    \]
    \[
    \sup_{0 < t < 1}|\hat{S}_{n}(t) - W(t)| \overset{p}{\longrightarrow} 0.
    \]
\end{theorem}

Finally, we need a bound on the maximum deviation of Brownian motion.

\begin{proposition}[Eq 8.20 ~\cite{MR1700749}]
    \label{p:foldedNormal}
    Let $(W(t))_{0 \leq t \leq 1}$ be standard Brownian motion. Then
    \[
    \prob{\sup_{0\leq t \leq 1} W(t) \geq x} = \frac{2}{\sqrt{2\pi}} \int_{x}^{\infty} e^{-u^2/2} \; du.
    \]
\end{proposition}

\subsection{Permutons}
A Borel probability measure $\mu$ on $[0,1]^2$ is called a \emph{permuton} if $\mu([0,1]\times [x,y]) = x-y = \mu([x,y]\times[0,1])$ for all $0 \leq x \leq y \leq 1$. First defined in ~\cite{hoppen2013limits} under the name ``limiting permutation", a permuton describes a limit of a permutation sequence via convergence of permutation pattern densities.

\subsection{Permutation patterns and U-statistics}
Fix $\pi \in \Sym_{n}$. 
We say the $k$-tuple of distinct indices $I = (i_1 < i_2 < \hdots < i_{k})$ forms the \emph{pattern} $\sigma \in \Sym_{k}$ in $\pi$ if $\pi(I)$ is order-isomorphic to $\sigma$. 

\begin{example}
    Let $\pi = 43512$. The indices $1$, $3$, and $4$ form the pattern $231$.
\end{example}

\begin{remark}
    This paper cares about patterns indexed by entries involved rather than the indices. It is clear to see that the indices $I$ form a pattern $\sigma$ in $\pi$ iff the corresponding entries form the pattern $\sigma^{-1}$ in $\pi^{-1}$.
\end{remark}

Let $X_1, X_2, \hdots, X_{n}$ be a sequence of i.i.d.\! 
random variables in a measurable space $S$. 
Let $f: S^{k} \rightarrow \mathbb{R}$ be a measurable function.
Define the function $U_{n}$ as follows
\[
    U_{n}(f) = \sum_{1 \leq i_{1} < \hdots i_{k} \leq n} f(X_{i_{1}}, \hdots, X_{i_{k}}) 
\]
We call $U_{n}$ a \emph{U-statistic}.

A notable class of U-statistics is the number of patterns $\sigma$ that occur in a uniformly random permutation $\pi$.
Indeed, by setting $X_{i}$ to be the uniform random variable with the real range $[0,1]$ (not as a set of integers) a permutation $\pi$ is induced by the relative orderings of $X_{i}$. 
Then, letting $f(i_1,i_2,\hdots,i_k)$ be the indicator function of whether $(i_1, \hdots, i_k)$ forms the pattern $\sigma$ in $\pi$, we see that pattern occurrence is a $U$-statistic.

U-statistics will only be used in Section 5. A result from ~\cite{janson2018renewal} implies the asymptotic normality of U-statistics.

\begin{theorem}[Corollary 3.5 ~\cite{janson2018renewal}]
\label{t:janson}
    Suppose that $f(X_{1}, \hdots, X_{k}) \in L^2$. Then , as $n \rightarrow \infty$,
    \[
        \frac{U_{n} - \mathbb{E}(U_{n})}{n^{k-1/2}} \overset{d}{\rightarrow} N(0,\sigma^2),
    \]
    where
    \[
        \sigma^2 = \lim_{n \rightarrow \infty} \frac{Var(U_{n})}{n^{2d-1}}.
    \]
\end{theorem}

\section{Maj Sampler}

 \begin{proof}[Proof of Theorem ~\ref{t:introsampler}]
 	Let $\pi \in \Sym_n$.
 	We will show that if given $G = (G_{i})_{i=1}^{n} \sim \geo{1-q}^n$, then $\Gamma_{n}(G) = \pi$ occurs with probability $q^{maj(\pi)}/[n]_q!$.
 	In order to obtain $\pi$ through $\Gamma_{n}$, we need the sequence $(G_{\pi(i)})_{i \in [n]}$ to be in weakly descending order.
 	In addition, if $i \in Des(\pi)$, then we require the extra condition that $G_{\pi(i)} > G_{\pi(i+1)}$.
        We will denote this conditionally strict inequality by $\cgeq{i}$.
 	For $k > 1$, let $A_k$ be the event that $G_{\pi(1)} \cgeq{1} \hdots \cgeq{k{-}1} G_{\pi(k)}$.
    As $A_{n-1} \supseteq A_{n}$, we condition $A_{n}$ as follows:
 	\begin{align}
 		\prob{A_n}  & = \prob{A_{n-1} \bigg\lvert \min_{i \in [n-1]}G_{\pi(i)} \cgeq{n{-}1} G_{\pi(n)}} \prob{\min_{i \in [n{-}1]}G_{\pi(i)} \cgeq{n-1} G_{\pi(n)}}.
 	\end{align}
 
For the first factor, $\underset{i \in [n-1]}{\min}G_{\pi(i)} \cgeq{n-1} G_{\pi(n)}$ is equivalent to $G_{\pi(i)} \cgeq{n-1} G_{\pi(n)}$ for all $i \in [n-1]$.
Thus, using Proposition \ref{prop:memory2} to obtain (3.5),
\begin{align}
\mathbb{P}\Big(A_{n-1} \bigmid & \min_{i \in [n-1]}G_{\pi(i)} \cgeq{n{-}1} G_{\pi(n)} \Big) \\
& = \prob{A_{n-1} \bigmid  \forall i \in [n-1], G_{\pi(i)} \cgeq{n{-}1} G_{\pi(n)}}\\
& = \sum_{k \geq 0} \prob{A_{n-1} \bigmid  \forall i \in [n-1], G_{\pi(i)} \cgeq{n-1} k} \prob{G_{\pi(n)}=k}\\
& = \sum_{k \geq 0} \prob{A_{n-1}} \prob{G_{\pi(n)}=k}\\
& = \prob{A_{n-1}}.
\end{align}
The second factor follows by Lemma \ref{prop:race}
\[
	\prob{\min_{i \in [n-1]}G_{\pi(i)} \cgeq{n-1} G_{\pi(n)}} = 
	\begin{cases}
		\frac{1-q}{1-q^{n}} & n-1 \notin Des(\pi)\\
		\frac{q^{n-1}(1-q)}{1-q^{n}} & n-1 \in Des(\pi).
	\end{cases}
\]
Together, this gives
\[
	\prob{A_n} = \frac{q^{(n-1)\mathbf{1}_{\{n-1 \in Des(\pi)\}}} }{[n]_q} \prob{A_{n-1}}
\]
Repeating the above argument until $\prob{A_2} = \dfrac{q^{\mathbf{1}_{\{1 \in Des(\pi)\}}}}{1+q}$ is reached, we obtain
\[
	\prob{\pi} = \prob{A_n} = \frac{q^{maj(\pi)}}{[n]_q!}.
\]
This completes the proof.
\end{proof}

As $\Gamma_{n}$ is not injective, $G$ cannot be recovered from $\pi$. 
In other words, information is lost when converting $\geo{1-q}^{n}$ to $\Maj$.
As such, properties about $\Maj$ may (and will) be better understood by analyzing $\geo{1-q}^n$ rather than $\Maj$ directly.
So from now on, whenever $\pi$ follows a major index distribution, there will be an underlying $G$ with the implicit assumption that $\pi = \Gamma_{n}(G)$.

\section{Limit Behavior}

Let $\mathcal{G}_{k} = \{i \mid G_{i} = k\}$. The sets $\mathcal{G}_{\leq k} = \{i \mid G_{i} \leq k\}$ and its variants are defined in an analogous manner.
Let $L_{i}: [0,n] \rightarrow [0,n]$ be a family of lattice walks defined as follows.
\begin{enumerate}
    \item $L_{0}(0) = n-|\mathcal{G}_{0}|$.
    \item $L_{i}(0) = L_{i-1}(0) - |\mathcal{G}_{i}|$ for all $i > 0$.
    \item $L_{i}(k) = 
    \begin{cases}
        L_{i}(k-1) & \text{ if $G_{k} = i$}\\
        L_{i}(k-1)+1 & \text{ otherwise}
    \end{cases}$.
\end{enumerate}

This family of lattice walks is well-defined as one could define $L_{0}$ and iteratively define the rest.
Condition (3) ensures that among all of these lattice paths there are exactly $n$ upright steps.
Combined with $(1)$ and $(2)$, this ensures the codomain falls within $[0,n]$.

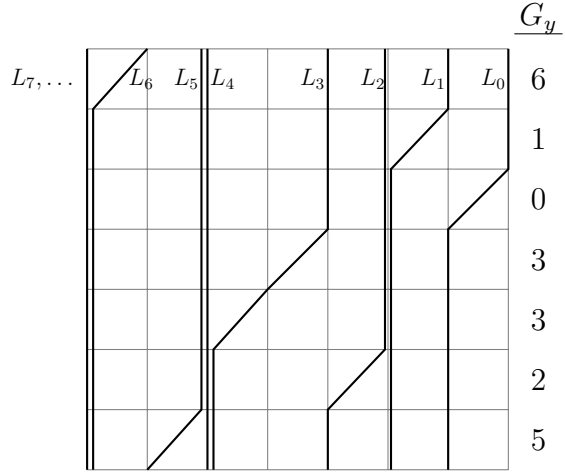
\begin{figure}
    \centering
    \begin{tikzpicture}[scale = 0.8]
            \draw[gray,thin] (0,0) grid (7,7);
            \draw[black,thick] (7,7) -- (7,6) -- (7,5) -- (6,4) -- (6,3) -- (6,2) -- (6,1) -- (6,0);
            \node[scale=0.75] at (6.75,6.5) {$L_0$};
            \draw[black,thick] (6,7) -- (6,6) -- (5.05,5) -- (5.05,4) -- (5.05,3) -- (5.05,2) -- (5.05,1) -- (5.05,0);
            \node[scale=0.75] at (5.75,6.5) {$L_1$};
            \draw[black,thick] (4.95,7) -- (4.95,6) -- (4.95,5) -- (4.95,4) -- (4.95,3) -- (4.95,2) -- (4,1) -- (4,0);
            \node[scale=0.75] at (4.75,6.5) {$L_2$};
            \draw[black,thick] (4,7) -- (4,6) -- (4,5) -- (4,4) -- (3,3) -- (2.1,2) -- (2.1,1) -- (2.1,0);
            \node[scale=0.75] at (3.75,6.5) {$L_3$};
            \draw[black,thick] (2,7) -- (2,0);
            \node[scale=0.75] at (2.25,6.5) {$L_4$};
            \draw[black,thick] (1.9,7) -- (1.9,6) -- (1.9,5) -- (1.9,4) -- (1.9,3) -- (1.9,2) -- (1.9,1) -- (1,0);
            \node[scale=0.75] at (1.65,6.5) {$L_5$};
            \draw[black,thick] (1,7) -- (0.1,6) -- (0.1,5) -- (0.1,4) -- (0.1,3) -- (0.1,2) -- (0.1,1) -- (0.1,0);
            \node[scale=0.75] at (0.9,6.5) {$L_6$};
            \draw[black,thick] (0,7) -- (0,0);
            \node[scale=0.75] at (-0.75,6.5) {$L_7, \hdots$};

            \node at (7.5,7.5) {$G_{y}$};
            \draw[black,thin] (7.1,7.25) -- (7.9,7.25);
            \node at (7.5,6.5) {$6$};
            \node at (7.5,5.5) {$1$};
            \node at (7.5,4.5) {$0$};
            \node at (7.5,3.5) {$3$};
            \node at (7.5,2.5) {$3$};
            \node at (7.5,1.5) {$2$};
            \node at (7.5,0.5) {$5$};
        \end{tikzpicture}
    \caption{The construction of the infinite sequence of lattice paths.
    Lattice walks can intersect but are drawn to look non-intersecting for clarity.
    The $G_{y}'s$ encode which lattice path gets an upright step at that height.}
    \label{fig:latticePaths}
\end{figure}

\begin{proof}[Proof of Theorem ~\ref{p:introApplyDonsker}]
Fix $i \geq 0$.
 The lattice walk $L_{i}(y) - L_{i}(0)$ is a simple random walk with $n$ steps. 
 That is, $L_{i}(y) - L_{i}(0) = X_{1} + X_{2} + \hdots + X_{n}$, where each $X_{i}$ is a Bernoulli random variable such that $X_{i} = 1$ with probability $p_{i}$.  
 Thus, $\frac{X_{i} - p_{i}}{\sqrt{p_{i}}}$ are i.i.d.\! random variables with mean $0$ and variance $1$.
 Applying Donsker's Theorem completes the proof.
\end{proof}

It is important to note that for $\pi = \Gamma_{n}(G)$, the points $(i, \pi(i))$ each lie on at least one of the lattice paths $(L_{i})$.
In particular, the end of any upright step is precisely one of these ($i, \pi(i)$).
So, understanding the trajectory of the lattice paths gives us information about the distribution of these points.
Thus, we can determine the permuton associated with major index distribution.

\begin{theorem}
The support of the permuton $\mu: [0,1]^{2} \rightarrow [0,1]$ associated to major index distribution is the union of line segments starting at $(q^{i+1},0)$ and ending at $(q^{i}, 1)$ for some integer $i \geq 0$.
\end{theorem}
\begin{proof}
    We first show that each $\mathcal{L}_{i}(t) = \dfrac{L_{i}(\lfloor nt \rfloor)}{n}$ will converge in probability to the line segment starting at $(q^{i+1},0)$ and ending at $(q^{i},1)$.
    Then, we will show that this countable union of lattice walks converges in probability to the desired union of line segments.

    Fix $\epsilon > 0$ and let $i \in \N$. 
    From Theorem ~\ref{p:introApplyDonsker}, we have that 
    \begin{equation}
        \left( \frac
        {L_{i}(\lfloor nt \rfloor) - p_{i}t - L_{i}(0)}
        {\sqrt{np_{i}(1-p_{i})}} \right)_{0 \leq t \leq 1} \overset{dist.}{\longrightarrow} (W(t))_{0 \leq t \leq 1},
    \end{equation}
    By applying Theorem ~\ref{p:donskerBounds}, we can find an $\hat{L}_{i}(t)$ and standard Brownian motion $W(t)$ such that $\hat{L}_{i}(t)$ converges to $W(t)$ in probability, and 
    \begin{equation}
        \hat{L}_{i}(t) \overset{\mbox{dist.}}{=} \sqrt{n} \mathcal{L}_{i}(t).
    \end{equation}

    By Proposition ~\ref{p:foldedNormal}, the supremum (and infimum due to symmetry) of $W(t)$ follows a distribution with a rapidly decaying tail. More precisely,
    \begin{equation} \label{borwnianToLine}
        \prob{\underset{0 \leq t \leq 1}{\sup} |W(t)| \geq \dfrac{\epsilon}{\sqrt{n}}} \leq \frac{4}{\sqrt{2\pi}} \int_{\epsilon/\sqrt{n}}^{\infty} e^{\frac{-u^2}{2}}\; du = o(n^{-1/2})
    \end{equation}
    Combining these two facts gets
    \begin{align}
        & \prob{\underset{0 \leq t \leq 1}{\sup}|\mathcal{L}_{i}(t) - \E[\mathcal{L}_{i}(t)]| > \epsilon } \\
        & = \prob{\underset{0 \leq t \leq 1}{\sup}|n^{-1/2} \hat{L}_{i}(t) - \E[n^{-1/2} \hat{L}_{i}(t)]| > \epsilon } \\
        & \leq \prob{ \underset{0 \leq t \leq 1}{\sup}|\mathcal{L}_{i}(t) - \E[n^{-1/2} \hat{L}_{i}(t)] - n^{-1/2}W(t)| > \frac{\epsilon}{2}} + \prob{\underset{0 \leq t \leq 1}{\sup}|n^{-1/2}W(t)| > \frac{\epsilon}{2}}\\
        & = o(n^{-1/2}).
    \end{align}

    Now, we need to show that $\E[\mathcal{L}_{i}(t)] = \frac{1}{n}(p_{i}t + L_{i}(0) )$ converges in probability to the line segment starting at $(q^{i+1},0)$ and ending at $(q^{i},1)$.
    By the construction of the lattice paths, $L_{i}(y)$ is equal to the cardinality of the disjoint union of the sets $\{ k \leq y \mid G_{k} < i\}$ and $\{k > y \mid G_{k} \leq i\}$.
    Both of these cardinalities are independent and follow a binomial distribution with success probability $q_{i+1}$ and $q_{i}$ respectively. 
    Together, $\mathbb{E}(\frac{1}{n}L_{i}(y))$ has mean $p_{i+1}(y/n) + q^{i+1}$.
    This describes the desired line segment.
    In addition, note that the variance of $\frac{1}{n}L_{i}(y)$ is the sum of the variances of these binomial distributions, which is $O(1/\sqrt{n})$. 
    Thus, we obtain via triangle inequality
    \[
        \prob{ \max_{i \in [n]} \left\lvert \frac{L_{i}( t ) - p_{i+1}y - q^{i+1}n}{n} \right\rvert > \epsilon } \leq O(1/\sqrt{n}).
    \]
    Therefore, $\mathcal{L}_{i}(t)$ converges in probability to the desired line segment.
    
    As the deviation from the expectation vanishes to 0, take $n$ large enough such that w.h.p., $L_{0}, L_{1}, \hdots, L_{i}$ are within $\epsilon$ of their corresponding line segments, where $i$ satisfies $q_{i+1} < \epsilon$. 
    Then w.h.p., all other walks $L_{i+1}, L_{i+2}, \hdots$ must take values less than $L_{i}(0) < q^{i+1} + \epsilon < 2\epsilon$ .
    Thus, all lattice paths fall within $2\epsilon$ of their corresponding line segment.
    Therefore, w.h.p., a permutation selected under major index distribution will be arbitrarily close to the support described in the hypothesis. 
\end{proof}

\section{Pattern Normality}

With the establishment of a permuton, pattern density for major index distribution is known to converge as $n \rightarrow \infty$.
This section proves a stronger result: the pattern densities for a fixed pattern $\sigma \in \Sym_{k}$ is constant once $n \geq k$.
In addition, pattern occurrence can be shown to be asymptotically normal as well.

\begin{proposition}
\label{p:pattern}
	Let $\pi \sim \Maj$. For $\pi^{-1}$, the probability that set of indices $I = (i_1 < i_2 < \hdots < i_k)$ form the pattern $\sigma^{-1} \in \Sym_k$ is 
 \[
 \frac{q^{\maj(\sigma)}}{[k]_q!}
 \]
\end{proposition}
\begin{proof}
	Set $\pi = \Gamma_{n}(G)$. In order for $\pi^{-1}(I)$ to be order isomorphic to $\sigma^{-1}$, we need 
 \[
 G_{i_{\sigma(1)}} \cgeq{1} G_{i_{\sigma(2)}} \cgeq{2} \hdots \cgeq{k-1} 
 G_{i_{\sigma(k)}}
 \]
 such that $\cgeq{j}$ is a strict inequality if $j \in Des(\sigma)$ and weak otherwise.
The proof then follows from the proof of Theorem ~\ref{t:introsampler}
\end{proof}

Let $T_{\sigma}(\pi)$ denote the number of occurrences of the pattern of $\sigma \in \Sym_{k}$ in the permutation $\pi \in \Sym_{n}$. 
The expectation of $T_{\sigma}(\Maj)$ follows from Proposition~\ref{p:pattern} and linearity of expectation.
Computing its variance requires tracking two copies of $\sigma$ in the $\Maj$.
To accomplish this, we use a variant of Stanley's Shuffling Theorem.

Let $\pi = \pi_1 \hdots \pi_{m}$ and $\sigma = \sigma_1 \hdots \sigma_{n}$ be permutations on disjoint values whose union is $[m+n]$. 
The \emph{shuffle} of $\pi$ and $\sigma$, denoted $\pi \shuffle \sigma$, is the set of permutations $\tau$ of $[m+n]$ such that $\pi$ and $\sigma$ appear as subsequences in $\tau$.

\begin{proposition}[Stanley's Shuffling Theorem, Th 1.1~\cite{stanley1972ordered}]
\label{p:stanley}
Let $\pi$ and $\sigma$ be permutations on disjoint values whose union is $[m+n]$.
Then
\[
\sum_{\tau \in \pi \shuffle \sigma} q^{\maj{\tau}} = \frac{[m+n]_q!}{[m]_q ! [n]_q !} q^{\maj(\pi) + \maj(\sigma)}.
\]
\end{proposition}

\begin{theorem}\label{t:pattern}
    
    For $\pi \sim \Maj$, then 
    \begin{enumerate} [label = (\roman*)]
        \item $\mathbb{E}(T_{\sigma}(\pi)) = \binom{n}{k} \frac{q^{\maj(\sigma)}}{[k]_q!} = \Theta(n^k).$
        \item $\mathrm{Var}\left( T_{\sigma}(\pi) \right) = O(n^{2k-1}).$
    \end{enumerate}
\end{theorem}
\begin{proof}
    Proposition ~\ref{p:pattern} and linearity of expectation makes (i) trivial.
    For (ii), we will show that the variance has no terms of order $n^{2k}$ or greater. 
    First note that $T_{\sigma}(\pi)^{2}$ is a sum of pattern occurrences $\tau$, where $\tau$ can be expressed as a (not necessarily disjoint) union of two copies of $\sigma$.
    As $\E[T_{\tau}(\pi)] = O(n^{|\tau|})$, the only $\tau$'s that contribute terms of $n^{2k}$ or higher are those that have length $2k$.
    Thus, we only need to consider such $\tau \in \Sym_{2k}$, which must be a disjoint union of two copies of $\sigma$. By counting the number of such $\tau$ and calculating the expected number of each one in $\pi$, we obtain
    \[
        \mathbb{E}(T_{\sigma}(\pi)^2) = \sum_{S \in \binom{[2k]}{k}} \sum_{\tau \in \sigma(S) \shuffle \sigma(S^C)} \binom{n}{2k} \frac{q^{\maj(\tau)}}{[2k]_q!} + O(n^{2k-1}).
    \]
    By using Stanley's Shuffling Theorem , we obtain
    \begin{align*}
        \mathbb{E}(T_{\sigma}(\pi)^2) & = \binom{2k}{k} \frac{[2k]_q!}{([k]_q!)^2} \binom{n}{2k} \frac{q^{2\maj(\sigma)}}{[2k]_q!} + O(n^{2k-1})\\
        & = \frac{q^{2\maj(\sigma)}}{(k!)^2([k]_{q}!)^2}n^{2k} + O(n^{2k-1}).
    \end{align*}
    Finally, note that
    \begin{align*}
        \mathbb{E}(T_{\sigma}(\pi))^2 & = \binom{n}{k}^2 \frac{q^{2\maj(\sigma)}}{([k]_q!)^2}\\
        & = \frac{q^{2\maj(\sigma)}}{([k]_q!)^2 (k!)^2 }n^{2k} + O(n^{2k-1}).
    \end{align*}
    has the same leading term as $\mathbb{E}(T_{\sigma}(\pi)^2)$.
    Thus, the highest order term of $\mathrm{Var}(T_{\sigma}(\pi))$ is at most order $n^{2k-1}$.
\end{proof}

 Note that $\geo{1-q}^{n}$ is an i.i.d.\!\! sequence of random variables and that pattern occurrences of $\pi^{-1}$ is a U-statistic of this sequence.
Thus, Theorem ~\ref{t:janson} directly leads to the following corollary.
\begin{corollary}
Pattern occurrence for $\Maj$ is asymptotically normal (or degenerate).
\end{corollary}

\begin{remark}
    As major index distribution deforms to the uniform distribution as $q \rightarrow 1$, we know that the $n^{2k-1}$ term for variance can be written as a nonnegative polynomial in $q$ that is not identically zero.
    However, it stands to reason some combinations of $q$ and $\sigma$ may reduce to a lower degree polynomial in $n$.
    This is the only case where the corollary above would give a degenerate distribution.
\end{remark}

\section{Fixed points and 2-cycles}
Using the lattice walk sequence defined in Section 3,
the cycle structure of major index distribution can be found.
For fixed points, this rests on two ideas.
\begin{itemize}
    \item Each $L_{k}$ crosses the $y=x$ line exactly once.
    \item The number of steps that $L_{k}$ stays on the $y=x$ line is roughly a geometric random variable.
\end{itemize}

This is done in a more rigorous fashion below.
Recall from Section 1 that $c_p = c_p(\pi)$ is the number of $p$-cycles of a permutation $\pi$, and $p_{i} = q^{i}(1-q)$.

\begin{proposition}
We have
\[
\E[c_1(\Maj)] = \sum_{\ell \geq 1} \frac{(1-q)^{\ell-1}}{[\ell]_q},
\]
\[
\mathrm{Var}[c_{1}(\Maj)] \overset{n \rightarrow \infty}{\longrightarrow} \sum_{\ell \geq 1} \frac{\ell(1-q)^{\ell-1}}{[\ell]_q}.
\]
\end{proposition}

\begin{proof}
    Consider $G \sim \geo{1-q}^{n}$ and $\pi = \Gamma_{n}(G)$.
    We will find the distribution of fixed points for each $\mathcal{G}_{i}$.
    
    Let $D_{i} = D_{i}(G) = \{i \in \mathcal{G}_{i} \mid \pi(i) = i\}$.
    As each $L_{i}$ consists of vertical and diagonal steps, $D_{i}$ must be a consecutive set of integers.

    \textbf{Expectation:} Fix $i \in \N$, and let $k \geq 0$. 
    Let $A_{(k,i)} = \{ G \in \mathbb{Z}_{\geq 0}^{n} \mid |D_{i}| \geq k\}$.
    Define $\phi_{(k,i)}: A_{(k,i)} \rightarrow \mathbb{Z}_{\geq 0}^{n-k}$ by removing the smallest $k$ elements of $D_{i}$.
    Explicitly, if $a = \min D_{i}$, then $\phi_{k}(G) = G' = (G'_{\ell})_{\ell = 1}^{n-k}$ where
    \[
        G'_{\ell} = \begin{cases}
            G_{\ell} & \ell < a\\
            G_{\ell + k} & \ell \geq a
        \end{cases}
    \]
    In terms of $\pi$, the elements $a, a+1, \hdots, a+k-1$ are removed and all remaining integers larger than $a+k$ are shifted down accordingly.
    Note that both $|D'_{i}| = |D_{i}| - k$ and $\phi_{k}$ is reversible.
    Considering $G' \in \mathbb{Z}_{\geq 0}^{n-k}$, there is a minimal element $a$ such that the lines $y=x$ and $L_{i}$ intersect at $(a,a)$.
    Simply set
    \[
        G_{\ell} = \begin{cases}
            G'_{\ell} & \ell < a\\
            i & a \leq \ell < a+k\\
            G'_{\ell - k} & \ell \geq a+k
        \end{cases}.
    \]
    This produces a $G \in A_{k}$ who maps to $G'$ via $\phi$.

    In addition, for all $G \in A_{k}$, we have
    \[
        \prob{G \mid G \sim \geo{1-q}^{n} } = (q^{i}(1-q))^k \prob{G' \mid G' \sim \geo{1-q}^{n-k}}.
    \]
    Thus, $\prob{G \in A_{(k,i)}} = p_{i}^{k}$, and so $|D_{i}| \geq k$ with probability $p_{i}^{k}$ when $k \leq n$ and $0$ otherwise.
    This implies that $|D_{i}| \sim \min( \geo{1-p_{i}}, n)$.
    
    Thus, the expected number of fixed points would be
    \begin{align*}
        \sum_{i \geq 0} \E[|D_i|] &  = \sum_{i \geq 0} \sum_{\ell \geq 1} \prob{D_i > \ell}
        = \sum_{i \geq 0} \sum_{\ell=1}^{n} q^{\ell i}(1-q)^{\ell}
        = \sum_{\ell=1}^{n} \frac{(1-q)^{\ell}}{1-q^{\ell}}\\
        & = \sum_{\ell=1}^{n} \frac{(1-q)^{\ell-1}}{[\ell]_q}.
    \end{align*}
    
    \textbf{Variance:} Fix $i < j \in \N$ and let $k, \ell \geq 0$.
    Similar to the previous argument, we will find the probability that $G$ is in the set $B_{k,\ell} = A_{(k,i)} \cap A_{(\ell,j)}$. Let $\phi_{(k,i)}$ be the same function as the previous argument that maps $G$ into $G'$. 
    As $j > i$, every element of $D_{j}(G)$ is smaller than every element of $D_{i}(G)$, so we have that $D_{j}(G') = D_{j}(G)$.
    Define $\phi'_{(\ell,j)}: \phi_{(k,i)}(B_{k,\ell}) \rightarrow \mathbb{Z}_{\geq 0}^{n-k-\ell}$ be obtained by removing the $\ell$ smallest elements of $D_{j}(G')$ and shifting the larger elements accordingly, just as $\phi_{(k,i)}$ does.
    It is clear that $\phi'_{(\ell,j)}$ is reversible by the same argument that $\phi_{(k,i)}$ is reversible.
    As the composition of reversible operations are reversible, we find that for all $G \in B_{k,\ell}$, 
    \[
        \prob{G \mid G \sim \geo{1-q}^{n} } = p_{i}^{k} p_{j}^{\ell} \prob{G' \mid G' \sim \geo{1-q}^{n-k}}.
    \]
    Explicitly, this means that the joint survival function of $|D_{i}|$ and $|D_{j}|$ is
    \[
        S(x,y) = \prob{|D_{i}| \geq k, |D_{j}| \geq \ell} = 
        \begin{cases}
            p_{i}^{k} p_{j}^{\ell} & k + \ell \leq n\\
            0 & \mbox{else}
        \end{cases}
    \]
    As $n \rightarrow \infty$, the joint survival function converges uniformly to that of two independent geometric random variables. 
    This establishes that $\cov(|D_{i}|, |D_{j}|)$ vanishes as $n \rightarrow \infty$ if $i \neq j$.
    Using a crude bound, we bound the covariances (and variances) above by 
    \begin{align*}
        |\cov(|D_{i}|, |D_{j}|)| & \leq \sqrt{ \Var(|D_{i}|)\Var(|D_{j}|) }\\
        & \leq \sqrt{ \dfrac{p_{i}}{(1-p_{i})^2} \dfrac{p_{j}}{(1-p_{j})^2} }\\
        & = \dfrac{\sqrt{p_{i}p_{j}}}{(1-p_{i})(1-p_{j})}\\
        & \leq \sqrt{q}^{i+j}(1-q).
    \end{align*} 
    Since 
    $\sum_{i, j = 0}^{\infty} \sqrt{q}^{i+j}(1-q) = (1-q)(1-\sqrt{q})^{-2}$,
    we have that the covariances, when treated as a function on the domain $\N$, is dominated by the integrable function $g(i,j) = \sqrt{q}^{i+j}(1-q)$. Thus, by the dominated convergence theorem, we get the following
    \begin{align*}
    \lim_{n \rightarrow \infty} \Var(c_{1}(\pi)) & = \lim_{n \rightarrow \infty} \sum_{i,j \leq 0} \cov(D_{i},D_{j})\\
        & =\lim_{n \rightarrow \infty} \sum_{i \leq 0} \Var(D_{i})\\
        & = \sum_{i \geq 0} \frac{p_{i}}{(1-p_{i})^2}\\
        & = \sum_{i \geq 0} \sum_{\ell \geq 1} \ell p_{i}^{\ell}\\
        & = \sum_{i \geq 0} \sum_{\ell \geq 1} \ell q^{i\ell}(1-q)^{\ell}\\
        & = \sum_{\ell \geq 1} \frac{\ell}{(1-q^{\ell})}{(1-q)^{\ell}}\\
        & = \sum_{\ell \geq 1} \frac{\ell(1-q)^{\ell-1}}{[\ell]_{q}}.
    \end{align*}
    This completes the proof.
\end{proof}

For $0 < q < 1$, the limit of the expectation does not meet the variance. Thus, the distribution of fixed points in $\Maj$ is not Poisson. 

The expectation for $c_{1}$ can be rewritten into the following recurrence. 
Let $\mathrm{FP}(n,q) = \sum_{\pi \in \Sym_{n}} c_1( \pi)q^{\maj(\pi)}$.
It would be interesting to find a combinatorial interpretation for the following result.

\begin{corollary}
    For $0 < q \leq 1$ and $n \geq 1$,
    \[
        \mathrm{FP}(n+1,q) = [n+1]_{q} \mathrm{FP}(n,q) + \prod_{i=1}^{n+1} (1-q^i) 
    \]
\end{corollary}

For 2-cycles, a similar procedure can be followed except that a pair of runs needs to be considered instead of a single run.
\begin{proposition}
The expected number of 2-cycles under $\Maj$ is 
\[
\sum_{\ell = 1}^{\lfloor n/2 \rfloor} \frac{q^{\ell}(1-q)^{2\ell-2}}{[\ell]_q [2\ell]_q}.
\]
\end{proposition}

\begin{proof}
    Let $G \sim \geo{1-q}^{n}$.
    The indices of any 2-cycle forms an inversion.
    As elements of any $\mathcal{G}_{k}$ form an increasing run, elements of a 2-cycle must inhabit different $\mathcal{G}_{k}$'s. For a lattice walk $L$, define its transpose $L^{T}$ such that $L^{T}(b) = a$ if and only if $L(a) = b$.

    Fix $i, j$ such that $i < j$.
    On the lattice walks $L_{i}$ and $L_{j}$, let $a$ be the minimal integer such that there exists points $(a,b) \in L_{i}$ and $(b,a) \in L_{j}$.
    This can be seen by reflecting $L_{j}$ across the $y=x$ line to form $L_{j}^{T}$.
    Furthermore, $L_{i}$ takes steps with slope at least $1$ and $L_{j}^{T}$ takes steps with slope at most $1$.
    So, the intersection set between $L_{i}$ and $L_{j}^{T}$ are connected by consecutive diagonal steps.
    Let $D_{i,j}$ denote the number of consecutive diagonal steps necessary to form the intersection set (alternatively, this is one less than the size of the intersection set).
    We claim that $D_{i,j} \sim \min (\geo{1-p_{i}p_{j}}, \lfloor \frac{n}{2} \rfloor)$.

    Let $A_{k} = \{ G \in \mathbb{Z}_{\geq 0}^{n} \lvert D_{i,j} \geq k \}$.
    Define the function $\phi_{k}: A_{k} \rightarrow \mathbb{Z}_{\geq 0}^{n-2k}$ by removing the first $k$ diagonal steps counted by $D_{i,j}$.
    Explicitly, $A_{k}(G) = G' = (G'_{\ell})_{\ell=1}^{n-2k}$ where
    \[G'_{\ell} = 
        \begin{cases}
            G_{\ell}, & \ell \leq b\\
            G_{\ell + k}, & b < \ell \leq a-k\\
            G_{\ell + 2k}, & a-k < \ell \leq n-2k 
        \end{cases}.\]
    Let $a'$ be the minimal integer such that there exist a point $(a',b') \in L'_{i} \cap L_{j}^{'T}$.
    It can be shown that $(a',b) = (a'-k, b)$ by noting the following.
    Removing the diagonal steps from $L_{i}$ caused the point $(b,a) \in L_{j}$ to be shifted to $(b, a-k) \in L'_{j}$.
    Removing the diagonal steps from $L_{j}$ caused the point $(a,b) \in L_{i}$ to be shifted to $(a-k, b) \in L'_{i}$.
    Since $L'_{i} \cap L_{j}^{T'}$ is connected by diagonal steps, any intersection before $(a',b)$ here would imply an intersection earlier than $(a,b)$ in $L_{i} \cap L_{j}^{T}$.
    Thus, $a$ and $b$ can be found from the image alone.
    Therefore, $\phi_{k}$ is reversible.

    Finally, note that for $2k \leq n$
    \[
    \prob{D_{ij} \geq k} = \sum_{G \in A_{k}} \mathbb{P}_{n}(G) = \sum_{G \in A_{k}} q^{ik}(1-q)^{k}q^{jk}(1-q)^{k}\mathbb{P}_{n}(\phi_k{G}) = q^{(i+j)k}(1-q)^{2k},
    \]
    and 0 otherwise.
    This proves the earlier claim about the distribution of $D_{i,j}$. The expected number of 2-cycles can then be obtained by linearity of expectation
    \begin{align*}
        \sum_{i < j} \mathbb{E}(D_{ij}) & = \sum_{\ell = 1}^{\lfloor n/2 \rfloor} \sum_{i = 0}^{n} \sum_{j = i+1}^{n} \prob{D_{ij} \geq \ell}\\
        & = \sum_{\ell = 1}^{\lfloor n/2 \rfloor} \sum_{i = 0}^{n} \sum_{j = i+1}^{n} q^{(i+j)\ell}(1-q)^{2\ell}\\
        & = \sum_{\ell = 1}^{\lfloor n/2 \rfloor} (1-q)^{2\ell} \sum_{i = 0}^{n} q^{i\ell} \frac{q^{(i+1)\ell}}{1-q^{\ell}}\\
        & = \sum_{\ell = 1}^{\lfloor n/2 \rfloor} (1-q)^{2\ell} \frac{q^{\ell}}{(1-q^{2\ell})(1-q^{\ell})}\\
        & = \sum_{\ell = 1}^{\lfloor n/2 \rfloor} \frac{q^{\ell}(1-q)^{2\ell-2}}{[\ell]_q [2\ell]_q}.
    \end{align*}
    This completes the proof.
\end{proof}

\section{q-Plancherel Measure}

To recover the main result in ~\cite{feray2012asymptotics}, recall that $\lambda(\pi) = (\lambda_1(\pi), \lambda_2(\pi), \hdots)$ were defined such that $\lambda_1(\pi) +  \lambda_2(\pi) + \hdots + \lambda_i(\pi)$ is the maximum total length amongst all $i$-tuples of disjoint increasing subsequences.
In essence, our argument will rely on the fact that the rightmost runs of a permutation under $\Maj$ is dense enough that using anything other than the rightmost runs is inefficient in the limit. 
The following lemma will help leverage this idea.

\begin{lemma}
\label{l:boundblock}
    Fix $q \in (0,1)$ and $k \in \N$, and 
    let $G = (G_{j})_{j=1}^{m} \sim \geo{1-q}^{m}$.
    \begin{enumerate}
    [label = (\roman*)]
      \item Among all $\mathcal{G}_{i}$, the $k$ largest sets are $\mathcal{G}_{0}, \mathcal{G}_{1}, \hdots, \mathcal{G}_{k-1}$ with probability $1-O(e^{-m})$. 
      \item The expected number of $\mathcal{G}_{i}$ that are nonempty is $O(\log (m))$.
    \end{enumerate}
\end{lemma}

\begin{proof}
Fix $q$ and $k$ as above.

(i) Let $P$ be any value strictly between $p_{k-1}$ and $p_{k}$.
We will show that with high likelihood, $\mathcal{G}_{0}, \hdots, \mathcal{G}_{k-1}$ have cardinalities larger than $Pm$ while all other $\mathcal{G}_{i}$'s have cardinalities smaller than $Pm$.

Each $|\mathcal{G}_{i}|$ is a binomial distribution with $n$ trials with success rate $p_{i}$  
For $i \geq k$, it follows from the Chernoff bounds on the $\mathcal{G}_{i}$'s and elementary calculus that
\begin{align*}
\prob{|\mathcal{G}_{i}| > Pm} & \leq \left( \left( \frac{p_{i}}{P} \right)^{P} \left( \frac{1 - p_{i}}{1 - P} \right)^{1-P} \right)^{m} < 1.\\
\end{align*}
Let $C_{i} = \left( \frac{p_{i}}{P} \right)^{P} \left( \frac{1 - p_{i}}{1 - P} \right)^{1-P}$. Observe that
\[
\lim_{i \rightarrow \infty} \frac{C_{i+1}}{C_{i}} = q^{P} < 1.
\]
So, $\sum_{i \geq k} C_{i} < \infty$, and there exists a $j \geq k$ such that $\sum_{i \geq k} C_{i} < 0.5$.
Thus, we get
\begin{align*}
    \sum_{i \geq k} C_{i}^{m} & = C_{k}^{m} + C_{k+1}^{m} + \hdots + C_{j-1}^{m} + \sum_{i \geq j} C_{i}^{m}\\
    & \leq C_{k}^{m} + C_{k+1}^{m} + \hdots + C_{j-1}^{m} + \left( \sum_{i \geq j} C_{i} \right)^{m}\\
    & \leq C_{k}^{m} + C_{k+1}^{m} + \hdots + C_{j-1}^{m} + 0.5^{m}\\
    & = O(e^{-m}).
\end{align*}
For the case of $i < k$, we can use the identity $\prob{|\mathcal{G}_{i}| < Pm} = \prob{m - |\mathcal{G}_{i}| > (1-P)m}$ and get a similar result.

(ii) Observe that
\begin{align*}
        \mathbb{E}(\#\{\mathcal{G}_k \neq \emptyset\})
        & = \sum_{k=0}^{\infty} \prob{\mathcal{G}_k \neq \emptyset}\\
        & \leq \sum_{k=0}^{\infty} \min( \mathbb{E}(\mathcal{G}_k), 1)\\
        & = |\{k \in \mathbb{Z}_{\geq 0} \mid mq^{k}(1-q) \geq 1\}| + \sum_{mq^{k}(1-q) < 1} mq^{k}(1-q)\\
        & \leq \log_{q}(m(1-q)) + \sum_{\ell=0}^{\infty} q^{\ell}\\
        & = O(\log (m)).
    \end{align*}
    This completes the proof.
\end{proof}

Even with the rigidity from the limit shape of the major index distribution, increasing subsequences will stray off from their expected $\mathcal{G}_{i}$ for a small segment. 
Though this deviation is of a negligible order, it does make calculations less tractable. 
So, we create a more tractable class of subsequences that generalize increasing subsequences.

Let $\pi = \Maj = \Gamma_{n}(G)$.
Let a \emph{net} $\mathcal{A}$ of $n$ be any set partition of $[n]$ into contiguous subsets $A_{1}, A_{2}, \hdots, A_{m}$. 
Let the \emph{inner/outer width} of $\mathcal{A}$ be the smallest/largest cardinality of the parts of $\mathcal{A}$.
An (entry-wise) subsequence $\tau$ of $\pi$ is \emph{blocked with respect to $\mathcal{A}$} if the following properties holds:
\begin{enumerate}
    \item For all $i \in [m]$, $\tau$ when restricted to $A_{i}$ is an increasing subsequence.
    \item Let $k = |\{j \in \N \mid \mathcal{G}_{j} > 0\}|$. With at most $k$ exceptions, for $i \in [m]$, there exists $j_{i} \in \mathbb{Z}^{+}$ such that $\mathcal{G}_j \cap A_{i} = \tau \cap A_{i}$.
\end{enumerate}

In other words, a subsequence blocked with respect to $\mathcal{A}$ is effectively $\mathcal{G}_{j}$ when restricted to some $A_{i}$ in all but a small number of cases dependent on $G$. 
We will call these exceptional $A_{i}$ \emph{bad blocks}.
In the case where $\mathcal{A}$ is unambiguous, we will call $\tau$ a \emph{blocked} subsequence.
Note that due to the geometry of $\Maj$, an increasing subsequence can only ``jump" between $\mathcal{G}$'s at most $k$ times.
Thus, increasing subsequences are blocked with respect to $\mathcal{A}$, regardless of the choice of $\mathcal{A}$.
As a result, the length of the longest blocked subsequence is weakly larger than that of the longest increasing subsequence.
As such, define $\lambda^{*}(G) = (\lambda_1^*(G), \lambda_2^*(G), \hdots)$ such that $\lambda_{\leq i}^{*}(G) = \lambda_1^*(G) + \lambda_2^*(G) + \hdots + \lambda_{i}^{*}(G)$ to be the maximum length of $i$ disjoint blocked (w.r.t. $\mathcal{A}$) subsequences of $\pi = \Gamma(G)$. 

\begin{proposition}
\label{p:bounding}
    Fix $0<q<1$ and let $G = (G_{1}, G_{2}, \hdots, G_{n}) ~ \geo{1-q}^{n}$. 
    Let $\mathcal{A}_{n}$ be a fixed sequence of nets such that the inner and outer width of $\mathcal{A}_{n}$ is $\Theta(n^{\epsilon})$ for $0 < \epsilon < 1$. 
    For all $k \in \mathbb{Z}_{\geq 0}$ and $i \in \N$ , 
    \[
    \E \left[ \left( \lambda_{\leq i}^*(G) - |\mathcal{G}_{< i}| \right)^{k} \right] = O(n^{\epsilon k} \log(n)^k).
    \]
\end{proposition}

\begin{proof}
    Fix $q$, $k$, $\epsilon$, and $i$ as above.
    As $\mathcal{G}_{0}, \hdots, \mathcal{G}_{i-1}$ is an $i$-tuple of disjoint blocked subsequences, $\lambda_{\leq i}^*(G) - |\mathcal{G}_{\leq i-1}| \geq 0$.
    Thus, we only need to bound the moments above.
    Let $A_{1}, \hdots, A_{m}$ be the parts of $\mathcal{A}_{n}$.
    Necessarily, $m = \Theta(n^{1-\epsilon})$.
    Let $\tau_{i} = \tau_{i}(G)$ be a set of elements of $[n]$ such that $\tau_{i}$ can be partitioned into $i$ disjoint set of blocked subsequences of $\pi$ that witness $\lambda_{\leq i}^*(G)$. 
    Fix this $i$-disjoint set as $T$.
    Let $|\tau_{i}|_{j} = |\tau_{i} \cap A_{j}|$ and define $|\mathcal{G}_{< i}|_{j}$ similarly.
    Observe that
    \[
    \E \left[ \left( \lambda_{\leq i}^*(G) - |\mathcal{G}_{\leq i-1}| \right)^{k} \right]
     = \sum_{1 \leq j_{1}, \hdots, j_{k} \leq m} \E \left[ \prod_{\ell = 1}^{k} (|\tau_{i}|_{j_{\ell}} - |\mathcal{G}_{\leq i-1}|_{j_{\ell}}) \right].
    \]
    Typically, $\lambda^{*}_{< j}(G) = |\mathcal{G}_{\leq j-1}|$ in all but two cases. 
    The first is when the runs $\mathcal{G}_{0}, \hdots, \mathcal{G}_{i-1}$ are not the $i$ largest runs after being intersected with $B_{j_{\ell}}$, which can happen with probability $O(e^{-n^{\epsilon}})$ by Lemma ~\ref{l:boundblock}.
    The second is when $A_{j}$ is a bad block for some element of $T$.
    This gives us 
    \[
        \E \left[ \prod_{\ell = 1}^{k} (|\tau_{i}|_{j_{\ell}} - |\mathcal{G}_{\leq i-1}|_{j_{\ell}}) \right] \leq \sum_{K \subseteq [k]} P(K)O(n^{k\epsilon}e^{-(k - |K|)n^{\epsilon}}), 
    \]
    where $P(K)$ is the probability that $A_{j_{\ell}}$ is a bad block for each $\ell \in K$. Note that $j_{\ell}$'s need not be distinct.

    By Lemma ~\ref{l:boundblock}, we have that $\sum_{j = 1}^{m} P(j) \leq O(\log n)$.
    Thus, we have an upper bound for summing up all $\ell$-tuples of distinct indices
    \[
    \sum_{K \subseteq [m], |K| = \ell} P(K) \leq \left( \sum_{j=1}^{m} P(j) \right)^{\ell} = O(\log(n)^{k}).
    \]
    Among $k$-tuples of indices with repeated elements, each repeated element loses a factor of $\log$ and gains a factor that depends only on $k$. As $k$ is fixed, this term does not contribute towards the leading asymptotic.
    Thus, adding up all expectations together, we obtain
    \[
    \E \left[ \left( \lambda_{\leq i}^*(G) - |\mathcal{G}_{< i}| \right)^{k} \right] = O(n^{\epsilon k} \log (n)^k).
    \]
    This completes the proof.
\end{proof}

Now, we can recover the main result from ~\cite{feray2012asymptotics}.

\begin{proof}[Proof of Theorem ~\ref{intro:feray}]
Let $\pi \in \Sym_{n}$ be distributed under major index law. 
That is, $\pi = \Gamma_{n}(G)$ for $G \sim \geo{1-q}^{n}$.
Set $k$ and $i$ as in the earlier theorem.
By the definition of increasing subsequences and blocked subsequences, we have
\begin{equation}
    |\mathcal{G}_{<i}| \leq \lambda_{\leq i}(\pi) \leq \lambda_{\leq i}^{*}(G)
\end{equation}
As $\lambda_{i}(\pi) = \lambda_{\leq i}(\pi) - \lambda_{\leq i-1}(\pi)$  (where we treat $\lambda_{-1}(\pi) = 0$), we can bound the quantity as follows
\begin{equation}
\label{e:bounds}
    \frac{|\mathcal{G}_{<i}| - \lambda_{\leq i-1}^{*}(G)}{\sqrt{n}} \leq \frac{\lambda_{i}(\pi)}{\sqrt{n}} \leq \frac{\lambda_{\leq i}^{*}(G) - |\mathcal{G}_{<i-1}|}{\sqrt{n}}.
\end{equation}
Using Proposition~\ref{p:bounding} with $\epsilon < 0.5$, we have by methods of moments that for all $k$, 
\begin{equation}
\lim_{n} \dfrac{|\mathcal{G}_{<k}|}{\sqrt{n}} = \lim_{n} \dfrac{\lambda_{\leq k}^{*}(G)}{\sqrt{n}}
\end{equation}
Note that since $|\mathcal{G}_{<j}|$ is a binomial distribution with fixed success rate, $\lim_{n} \dfrac{|\mathcal{G}_{<k}|}{\sqrt{n}}$ is the normal distribution with expectation $\sum_{j < k} q^{j}(1-q)$.
Finally, taking the limit of ~\ref{e:bounds} and shifting the mean by $p_{i-1}$, we can conclude that
\begin{equation}
    Y_{i} = \lim_{n} \dfrac{\lambda_{i}(\pi) - p_{i-1}}{\sqrt{n}} = \lim_{n} \dfrac{|\mathcal{G}_{i-1}| - p_{i-1}}{\sqrt{n}} = \mathcal{N} \left(0, p_{i-1}(1-p_{i-1}) \right)
\end{equation}
This covers $\E[Y_{i}]$ and $\E[Y_{i}^{2}]$.
For the covariance, observe that we can also take the limit of $Y_{i} + Y_{j}$.

\begin{equation}
    Y_{i} + Y_{j} = \lim_{n} \dfrac{|\mathcal{G}_{i-1}| + |\mathcal{G}_{j-1}| - p_{i-1} - p_{j-1}}{\sqrt{n}} = \mathcal{N} \left(0, (p_{i-1}+p_{j-1}(1-p_{i-1}-p_{j-1}) \right)
\end{equation}

So, the the covariance would be

\begin{align}
    \cov(Y_{i}, Y_{j}) & = \dfrac{1}{2} \left( \Var(Y_{i} + Y_{j}) - \Var(Y_{i}) - \Var(Y_{j}) \right)\\
    & = \frac{1}{2} \left( (p_{i-1} + p_{j-1})(1-p_{i-1}-p_{j-1}) - p_{i-1}(1-p_{i-1}) - p_{j-1}(1-p_{j-1}) \right)\\
    & = \frac{1}{2} \left( - p_{i-1}p_{j-1} -p_{j-1}p_{i-1} \right)\\
    & = -(1-q)^2q^{i+j-2}.
\end{align}
This completes the proof.
\end{proof}

\section{Future directions}

Due to the simplicity of the sampler, it is natural to ask what if $\geo{1-q}$ is replaced with a different distribution.
Only distributions with some discreteness will give nontrivial results as any continuous distribution can be reduce to the uniform distribution on $\Sym_{n}$.
Of note, the case of uniform discrete distributions is effectively the study of random words of $n$ letters with a fixed alphabet.
In the case that the size of the alphabet is of $\Theta(n^{\alpha})$, this has been studied from the angle of representation theory for $\alpha \geq \frac{1}{2}$ in ~\cite{biane2001approximate} and $\alpha < \frac{1}{2}$ in ~\cite{feray2012asymptotics}. 

Another direction to take this is to have $q$ change with $n$.
Similar research of the longest increasing subsequence has been done with Mallows distribution by ~\cite{mueller2013length, starr2009thermodynamic}.
Based off of how Mallows developed, the next natural setting is to take $q = 1 - \Theta(n^{-\alpha})$ for $\alpha > 0$.
\begin{conjecture}
    For $q = 1 - \Theta(n^{-\alpha})$ for $0 < \alpha < \frac{1}{2}$, we have
    \[
    \lim_{n} \E \left[\dfrac{\LIS (\Maj)}{n^{1-\alpha} } \right] = (1-q). 
    \]
\end{conjecture}
The proof outline in Section 7 is expected to work, contingent on a modification of Lemma ~\ref{l:boundblock} holding.

For $\alpha > \frac{1}{2}$, it is expected that the disparity between the $|\mathcal{G}_{k}|$ are no longer relevant and the behavior is similar to that of the uniform distribution.
\begin{conjecture}
    For $q = 1 - \Theta(n^{-\alpha})$ for $\alpha > \frac{1}{2}$, we have
    \[
    \lim_{n} \E \left[\dfrac{\LIS (\Maj)}{\sqrt{n} }) \right] = 2. 
    \]
\end{conjecture}

The true interest lies within $\alpha = \frac{1}{2}$, where both the disparity of $|\mathcal{G}_{k}|$ and the uniform measure behavior are present.

\printbibliography
\end{document}